\documentclass[12pt,hyp,]{amsart}

\usepackage{rotating}
\usepackage{tikz}
\usepackage{tikz-cd}
\usepackage{subfig}
\usepackage{adjustbox}
\usepackage{cite}
\usepackage[colorlinks, linkcolor=blue, allcolors=blue]{hyperref}
\hypersetup{nesting=true,debug=true,naturalnames=true}
\usepackage[margin=1in]{geometry}

\usetikzlibrary{decorations.markings}

\let\<\langle
\let\>\rangle

\let\uml\"

\title{On The Virtual Cosmetic Surgery Conjecture}

\author{Keegan Boyle}  

\email{kboyle@uoregon.edu}


\newcommand{\Z}{{\mathbb{Z}}}
\newcommand{\R}{{\mathbb{R}}}
\newcommand{\Q}{{\mathbb{Q}}}
\newcommand{\C}{{\mathbb{C}}}
\newcommand{\vol}{{\mbox{Vol}}}

\newtheorem{lemma}{Lemma}
\newtheorem{proposition}{Proposition}
\newtheorem{corollary}{Corollary}
\newtheorem{theorem}{Theorem}
\newtheorem{conjecture}{Conjecture}

\theoremstyle{definition}
\newtheorem{definition}{Definition}[section]
\newtheorem{remark}[definition]{Remark}
\newtheorem{example}[definition]{Example}

\begin{document}

\begin{abstract}
Let $K$ be a knot in $S^3$, and $M$ and $M'$ be distinct Dehn surgeries along $K$. We investigate when $M$ covers $M'$. When $K$ is a torus knot, we provide a complete classification of such covers. When $K$ is a hyperbolic knot, we provide partial results in the direction of the conjecture that $M$ never covers $M'$.
\end{abstract}

\maketitle
\tableofcontents

\section{Introduction}

\subsection{Main Results}

Dehn surgery is an important method for constructing 3-manifolds. Extensive work has been done to understand this construction, but many elementary questions remain unresolved. For example, let $M$ be a closed oriented 3-manifold, $K$ a knot in $M$, and $\gamma$, $\gamma'$ surgery slopes along $K$. Denote by $M_{\gamma}(K)$ Dehn surgery on $K$ in $M$ along $\gamma$. One may ask when $M_{\gamma}(K)$ is homeomorphic to $M_{\gamma'}(K)$. In particular, the following conjecture regarding the uniqueness of Dehn surgery along knots has been around since at least 1991 \cite[Conjecture 6.1]{G}.

\begin{conjecture} [Cosmetic Surgery Conjecture]
\label{conj:csc}
If $M - K$ is not a solid torus and there exists an orientation preserving homeomorphism between $M_{\gamma}(K)$ and $M_{\gamma'}(K)$ then there exists a self-homeomorphism of $M - K$ taking $\gamma$ to $\gamma'$.
\end{conjecture}

Many partial results have been shown. For example, in 1990, Mathieu showed \cite{Ma} that the orientation preserving requirement is necessary by constructing an orientation reversing counterexample. See also \cite{BHW}. In 2015 Ni and Wu \cite{NW} proved that if surgery on $\gamma$ and $\gamma'$ provide a counterexample to the conjecture for a knot in $S^3$, then $\gamma = -\gamma'$. Perhaps most recently Jeon proved \cite{J} in 2016 that the conjecture is true for all but finitely many surgeries on each knot in a fairly general class of hyperbolic knots.

As a generalization of the cosmetic surgery question Lidman and Manolescu \cite[Question 1.15]{LM} asked when $M_{\gamma}(K)$ covers $M_{\gamma'}(K)$. Restricting to knots in $S^3$, we use the homological framing to write $\gamma$ as $p/q \in \Q$ with gcd$(p,q) = 1$. With this notation, a naive generalization of Conjecture \ref{conj:csc} for knots in $S^3$ might be

\begin{conjecture} [Virtual Cosmetic Surgery Conjecture]
\label{conj:vcsc}
If $K \subset S^3$ is not the unknot, $p'/q' \neq p/q \neq \infty$, and there exists a covering map of degree $d$ from $S^3_{\gamma}(K)$ to $S^3_{\gamma'}(K)$, then there exists a degree $d$ self-covering map of $S^3 - K$ taking the $p/q$ curve to the $p'/q'$ curve.
\end{conjecture}

\begin{remark}
The $p/q \neq \infty$ condition is necessary since there exist lens space surgeries on hyperbolic knots. We do not restrict to orientation preserving covers, since it is difficult to keep track of orientations. 
\end{remark}

This conjecture is false for torus knots $T(r,s)$ in $S^3$, see Examples \ref{ex:cover} and \ref{ex:cover2}, but we will classify counterexamples. In order to do so, we prove a structure theorem for covers between Seifert fiber spaces (see Proposition \ref{prop:pullback}), which reduces the question to classifying all covers between orbifolds with base space $S^2$ and 3 or fewer cone points (These are called \emph{small} Seifert fiber spaces, see section \ref{sec:realize}). 

\begin{theorem}
\label{thm:orb}
Let $S^2(a,b,c) \to S^2(a', b', c')$ be a degree $n >1$ cover of 2-orbifolds over $S^2$ with cone points of orders $a, b, c$, and $a', b', c'$ respectively. Then 
\begin{enumerate}
\item If $\frac{1}{a} + \frac{1}{b} + \frac{1}{c} < 1$, then $(a,b,c),(a',b',c'),n$ are one of the following up to reordering of $(a,b,c)$ and $(a',b',c')$, for some $x, y \in \Z$.
\[
\begin{array}{c|c|c||c|c|c}
(a,b,c) & (a',b',c') & n &(a,b,c)&(a',b',c')&n \\
\hline
 (x,x,y)& (2,x,2y)& 2 & (x,4x,4x)& (2,3,4x)& 6 \\
 (2,x,2x)& (2,3,2x)& 3 & (3,3,7)& (2,3,7)& 8 \\
 (x,x,x)& (3,3,x)& 3 & (2,7,7)& (2,3,7)& 9 \\
(3,x,3x)& (2,3,3x)& 4 & (3,8,8)& (2,3,8)& 10 \\
(x,2x,2x)& (2,4,2x)& 4 & (4,8,8)& (2,3,8)& 12 \\
  (x,x,x)& (2,3,2x)& 6 & (9,9,9)& (2,3,9)& 12 \\
 (4,4,5)& (2,4,5)& 6 & & &  \\

\end{array}
\]
\item If $\frac{1}{a} + \frac{1}{b} + \frac{1}{c} = 1$, then $(a,b,c),(a',b',c'),n$ are one of the following up to reordering of $(a,b,c)$ and $(a',b',c')$, where $n = x^2 + xy + y^2$ and $m = x^2 + y^2$ for some $x,y \in \Z$.
\[
\begin{array}{c|c|c}
(a,b,c) & (a',b',c') & n \\
\hline
 (2,3,6)& (2,3,6)& n \\
 (2,4,4)& (2,4,4)& m \\
 (3,3,3)& (3,3,3)& n \\
 (3,3,3)& (2,3,6)& 2n \\
\end{array}
\]

\item $\frac{1}{a} + \frac{1}{b} + \frac{1}{c} > 1$, then $(a,b,c),(a',b',c'),n$ are one of the following up to reordering of $(a,b,c)$ and $(a',b',c')$, for some $x,y \in \Z$.
\[
\begin{array}{c|c|c|c||c|c|c}
(a,b,c) & (a',b',c') & n& $conditions$ &(a,b,c)&(a',b',c')&n\\
\hline
(1, x, y) & (1, nx, ny) & n &&(2,3,3) & (2,3,4) & 2 \\
(1,d,d) & (2,2,x) & 2x/d & d|x &(2,2,3) & (2,3,4) & 4 \\
(2,2,d) & (2,2,x) & x/d & d|x &(2,3,3) & (2,3,5) & 5 \\
(1,d,d) & (2,3,3) & 12/d & d \in \{1,2,3\} &(2,2,5) & (2,3,5) & 6 \\
(1,d,d) & (2,3,4) & 24/d & d\in \{1,2,3,4\} &(2,2,3) & (2,3,5) & 10 \\
(1,d,d) & (2,3,5) & 60/d & d  \in \{1,2,3,5\} &&& \\

\end{array}
\]
\end{enumerate}
Furthermore, we construct all of the above covers.
\end{theorem}

\begin{remark}
It is interesting to note that many Seifert fibered surgeries on other knots are also known to be small, for example alternating hyperbolic knots \cite{IM}, and hence the covers between Seifert fibered surgeries on such knots are also understood through Theorem \ref{thm:orb}. 
\end{remark}
The covers in Theorem \ref{thm:orb} give counter examples to Conjecture \ref{conj:vcsc} for torus knots, but we provide a nice structure theorem in the cases where these exceptional covers do not occur.

\begin{theorem}
\label{thm:torusmain}
Let $r,s > 2$, $(r,s) \neq (3,4), (3,5),(4,5), (3, 7),$ or $ (3,8)$. Then $S^3_{p/q}(T(r,s))$ covers $S^3_{p'/q'}(T(r,s))$ if and only if all of the following hold.
\begin{enumerate}
\item $|rsq-p| = |rsq' - p'|$
\item $p | p'$
\item gcd$(p/p',rsq-p) = \mbox{gcd}(p/p', rs) = 1$
\item $p/p' \equiv 1 $ mod $rs$
\end{enumerate}
If these are satisfied, then the degree of the cover is $p'/p$.
\end{theorem}

One might hope that in this case Conjecture \ref{conj:vcsc} is satisfied, but in fact even covers over a fixed base orbifold can give counterexamples. See Example \ref{ex:cover2}.

In the case of hyperbolic knots, Mostow rigidity implies that there are no non-trivial self covers of the knot complements. In this case Conjecture \ref{conj:vcsc} would reduce to the cosmetic surgery conjecture on hyperbolic knots for trivial covers, and the following conjecture.

\begin{conjecture}[Hyperbolic Virtual Cosmetic Surgery Conjecture]
\label{conj:hyp}
If $p/q \neq p'/q' \in \Q$, then $S^3_{p/q}(K)$ does not non-trivially cover $S^3_{p'/q'}(K)$ for any hyperbolic knot $K$.
\end{conjecture}

An argument pointed out by a referee of a previous version shows that the following proposition, which is precisely stated later as Corollary \ref{cor:bounds}, is a consequence of \cite[Theorem 1.1]{FKP}.

\begin{proposition}
\label{prop:hyp}
Conjecture \ref{conj:hyp} is true for all but at most $32$ $p'/q'$ slopes on each hyperbolic knot $K \subset S^3$.
\end{proposition}

Focusing on low crossing number knots, some computations in SnapPy \cite{SnapPy} along with known information about exceptional surgeries on twist knots and pretzel knots give the following.

\begin{proposition}
\label{prop:hyp2}
Conjecture \ref{conj:hyp} is true for all hyperbolic knots with $8$ or fewer crossings.
\end{proposition}

\subsection{Outline of the Paper}
The organization of the paper is as follows. In section \ref{sec:background} we provide some background. In sections \ref{sec:lens} through \ref{sec:realize} we discuss the case of torus knots, proving Theorem \ref{thm:orb} in section \ref{sec:orbcov} and Theorem \ref{thm:torusmain} in section \ref{sec:realize}. In section \ref{sec:hyp} we discuss the case of hyperbolic knots, culminating in the proofs of Propositions \ref{prop:hyp} and \ref{prop:hyp2}.

\subsection{Acknowledgements}
I would like to thank the referee on a previous version for useful comments, Nathan Dunfield, Jessica Purcell, and Cameron Gordon for helpful conversations, and Robert Lipshitz for support, suggestions, and corrections. I would also like to thank Tye Lidman for pointing out an error in a previous version.

\section{Background \label{sec:background}}

All 3-manifolds are assumed compact, connected and orientable, although not oriented. For convience throughout, we will only work with non-trivial positive torus knots $T(r,s)$ with $r,s>0$.

We will use the notation $S^2(\alpha_1, \dots, \alpha_n)$ to mean the orbifold with underlying surface $S^2$, and $n$ cone points points with $\Z/\alpha_i\Z$ isotropy subgroups.  In the 1970s, Moser classified Dehn surgeries on torus knots: 
\begin{theorem}\cite[Theorem 1]{M}
\label{thm:moser}
Let $K$ be the $(r,s)$ torus knot, and $M$ be $S^3_{p/q}(K)$. Then
\begin{enumerate}
\item[(1)] If $|rsq-p| > 1$ then $M$ is a Seifert fiber space with base orbifold $S^2(r,s,|rsq-p|)$.
\item[(2)] If $|rsq-p| = 1$ then $M$ is the lens space $L(p, qs^2)$.
\item[(3)] If $rsq-p = 0$ then $M$ is $L(r,s) \# L(s,r)$. 
\end{enumerate}
\end{theorem}

Note that $L(-m,n)$ is understood to mean $L(m,-n)$ when $m > 0$, and that since $p/q = -p/(-q)$ give the same surgery, it can always be arranged that $rsq-p \geq 0$. Note that we are only considering manifolds up to orientation reversing homeomorphism. 

Let $M$ be an oriented Seifert fiber space with base orbifold $S^2(\alpha_1, \dots , \alpha_n)$ and Seifert invariants $b, \{(\alpha_i,\beta_i)\}$. For convenience we will not require the normalization $0 < \beta_i < \alpha_i$. We will use the standard notation
\[
\{b;(o_1,0);(\alpha_1, \beta_1), \dots ,(\alpha_n,\beta_n)\}.
\]
Throughout, we will omit the $(o_1,0)$ term, which indicates that the base orbifold is $S^2$ and that $M$ is orientable, since this will be true for all of our Seifert fiber spaces. For more information see \cite{JN}.
It will be useful to recall some facts about orbifold covers and Seifert fiber spaces. We use Thurston's definition of a covering map of orbifolds, see \cite[Chapter 13]{T}. 

\begin{definition}
The \emph{orbifold Euler characteristic} of a compact 2-dimensional orbifold $\Sigma$ with underlying manifold $S$, $r$ corner reflectors of orders $\{n_i\}$ and $s$ cone points of orders $\{m_j\}$ is 
\[
\chi(\Sigma) := \chi(S) - \frac{1}{2} \sum_{i=1}^{r}\bigg{(}1-\frac{1}{n_i} \bigg{)}- \sum_{j=1}^{s} \bigg{(}1-\frac{1}{m_j}\bigg{)}.
\]
\end{definition}
Note that by the Riemann-Hurwitz formula, $\chi(\Sigma)$ is multiplicative under finite covers. In the case at hand, suppose $S^2(a,b,c) \to S^2(a',b',c')$ is a covering space of degree $d$. Then 
\[
\chi(S^2) - \bigg{(}1-\frac{1}{a} \bigg{)} - \bigg{(}1-\frac{1}{b} \bigg{)} - \bigg{(}1-\frac{1}{c} \bigg{)} = d \bigg{(} \chi(S^2) - \bigg{(}1-\frac{1}{a'} \bigg{)} - \bigg{(} 1-\frac{1}{b'}\bigg{)} - \bigg{(} 1-\frac{1}{c'} \bigg{)} \bigg{)}.
\] 
More succinctly,
\begin{align}
\frac{1}{a} + \frac{1}{b} + \frac{1}{c} - 1 = d \bigg{(} \frac{1}{a'} + \frac{1}{b'} + \frac{1}{c'} - 1 \bigg{)}.
\end{align}
Additionally, looking at the preimages of the orbifold points $a', b',$ and $c'$, there is an obvious condition on $d$ which we will now describe. 

For any partition $\lambda_a = \{a_1, \dots a_n\}$ of $d$ where $a_i | a$, let $\lambda^a$ refer to the set $\{ a/a_1, \dots a/a_n\}$. Now observe that given a cover $S^2(a,b,c) \to S^2(a',b',c')$ of degree $d$, there exist partitions $\lambda_{a'}, \lambda_{b'}$ and $\lambda_{c'}$ of $d$ by divisors of $a', b'$, and $c'$ respectively so that the union $\lambda^{a'} \cup \lambda^{b'} \cup \lambda^{c'}$ consists entirely of 1s except for a single $a$, $b$, and $c$. We will refer to this as the \emph{partition condition} for orbifold covers. 

\begin{definition}
A \emph{Seifert neighborhood} of a fiber $\gamma$ in a Seifert fiber space is a fiber preserving and orientation preserving homeomorphism from a neighborhood of $\gamma$ to $I \times D^2 / \sim$ where $(0,z) \sim (1, e^{2 \pi i q/ p }z)$ for some pair of relatively prime integers $p$ and $q$, and the fibers are cycles of vertical fibers $I \times *$. Once such a homeomorphism is fixed we will refer to such a neighborhood as $N_{\frac{q}{p}}(\gamma)$.
\end{definition}
By definition a Seifert neighborhood exists for every fiber, and $p$ is the \emph{index} of the fiber. A fiber is \emph{regular} if $p = 1$ and \emph{singular} otherwise.

\begin{definition} Given a covering $f: \widetilde{M} \to M$, a \textit{pre-regular} fiber $\gamma \subset \widetilde{M}$ is a Seifert fiber of $\widetilde{M}$ such that $f(\gamma)$ is a regular fiber of $M$. A \textit{pre-singular} fiber $\gamma$ is one such that $f(\gamma)$ is a singular fiber of $M$.
\end{definition}

The following is a restatement of an observation in \cite{M}, which will be needed to discuss realizations of Seifert fiber spaces as surgeries on specific torus knots. We assume throughout that $r,s>0$.
\begin{lemma}
\label{lem:CRT}
Fix a torus knot $T(r,s)$. If $p/q$ surgery on $T(r,s)$ is a small Seifert fiber space, then the $b$ and $(\alpha_i,\beta_i)$ Seifert invariants are numerically determined by $r,s,p,$ and $q$.

\end{lemma}
\begin{proof} See \cite{M} or \cite[Section2.5]{GL}.
\end{proof}

\begin{definition}
A \emph{Seifert cover} is a covering map of Seifert fiber spaces which takes fibers to fibers.
\end{definition}

\section{Lens spaces and connect sums of lens spaces \label{sec:lens}}
In this section we will resolve Conjecture \ref{conj:vcsc} in the case when the base space is a lens space or a connect sum of lens spaces. That is, we consider case (2) in Theorem \ref{thm:moser}.
\begin{lemma}
Let $M$ and $M'$ be obtained from Dehn surgery on a torus knot $K$ which is not the unknot. Then if either $M$ or $M'$ is of type (2) in Moser's classification, then there is no covering map $f: M \to M'$.
\end{lemma}
\begin{proof}
On a non-trivial torus knot $T(p,q)$ there is a unique reducible surgery $S^3_{pq/1}(T(p,q))$ by Theorem \ref{thm:moser}. Indeed, all other surgeries Seifert fiber spaces over $S^2$ (and are not $S^2 \times S^1$, since $T(p,q)$ is non-trivial), and hence are irreducible. However, by the sphere theorem any cover of a reducible 3-manifold is reducible, since $\pi_2$ is preserved by covers.
\end{proof}

\begin{lemma}
\label{lem:lensspaces}
If $L(p,q)$ and $L(p',q')$ are lens spaces obtained from surgeries on the same torus knot, then $L(p,q)$ covers $L(p',q')$ if and only if $p$ divides $p'$.
\end{lemma}

\begin{proof}
The lens space $L(p',x)$ has a unique cover for each divisor $d$ of $p'$, and that cover is $L(p'/d,x)$, so the only if direction is clear. On the other hand, looking at which lens spaces are possible as surgeries on the same torus knot, we get from (2) in Theorem \ref{thm:moser} that gcd$(r,p') = 1$ and that $q'rs \equiv 1$ mod $p'$, after choosing $p',q'$ so that $rsq' +p' \geq 0$. Hence we can write $q's^2$ as $sr^{-1}$ mod $p'$. 

Now suppose that $L(p',x)$ and $L(p'/d,y)$ occur as $(p',q')$ and $(p,q)$ surgery respectively on the same torus knot, so that $x = q's^2$ and $y = qs^2$. Then $qrs \equiv \pm 1$ mod $p$ so that $x \equiv \pm sr^{-1}$ mod $p$ (and the same for $y$), giving $x \equiv \pm y$ mod $p$. Then by the classification of (unoriented) lens spaces $L(p'/d,y) \cong L(p'/d,x)$, and so $L(p'/d,y)$ covers $L(p',x)$. 
\end{proof} 

Since the only covers of lens spaces are lens spaces, this finishes the case where the base 3-manifold is a lens space.

\section{Covers of Seifert fiber spaces}

Throughout this section let $M$ be an orientable Seifert fiber space with the underlying surface of the base orbifold $S^2$, i.e. $M \cong \{b; (\alpha_1, \beta_1), \dots, (\alpha_n, \beta_n)\}$. Let $f:\widetilde{M} \to M$ be a covering map. Then there is an induced Seifert fiber structure on $\widetilde{M}$ where the fibers are the preimages of the fibers in $M$; see for example \cite[lemma 8.1]{JN}. In particular, there is a choice of Seifert fiber structure on $\widetilde{M}$ so that $f$ is a Seifert cover. Note however, that $\widetilde{M}$ may have other Seifert fiber structures for which $f$ is not even homotopic to a Seifert cover. Similar results to those in this section are observed in \cite[Section 2]{Hu}.

\begin{definition}
A \textit{fiberwise cover} is a Seifert cover $f: \widetilde{M} \to M$ for which the preimage of each fiber of $M$ is a single fiber of $\widetilde{M}$. 
\end{definition}

We will observe below that fiberwise covers induce an isomorphism between the base orbifolds.

\begin{definition}
A \textit{pullback cover} is a Seifert cover $f: \widetilde{M} \to M$ which induces a covering map $f_*: \widetilde{\Sigma} \to \Sigma$ of base orbifolds with deg$(f) = $ deg$(f_*)$. 
\end{definition}

\begin{remark}
The term \textit{pullback} is justified by the following proposition, which implies the universal property, and hence uniqueness, of such covers.
\end{remark}

\begin{proposition}
\label{prop:pullback}
Given a cover of Seifert fiber spaces $f: \widetilde{M} \to M$, $f$ factors as a composition of a fiberwise cover $f_2: \widetilde{M} \to \overline{M}$ and a pullback cover $f_1: \overline{M} \to M$. In particular, $f$ induces a covering map of base orbifolds $\widetilde{\Sigma} \to \Sigma$. This is notated as

\[
\begin{tikzcd}
S^1 \arrow{r}{\mbox{\tiny{deg}}(f_2)} \arrow{d} & S^1 \arrow{r}{\mbox{\tiny{id}}} \arrow{d} & S^1 \arrow{d} \\
\widetilde{M} \arrow{r}{f_2} \arrow{d} & \overline{M} \arrow{d} \arrow{r}{f_1} & M \arrow{d}{\rho} \\
\widetilde{\Sigma} \arrow{r}{\mbox{\tiny{id}}} & \widetilde{\Sigma} \arrow{r}{\mbox{\tiny{deg}}(f_1)} & \Sigma,
\end{tikzcd}
\]
where $\overline{M}$ is the pullback of the bottom right square, the columns are Seifert fibrations and the bottom row are the base orbifolds. The top left $S^1$ is a pre-regular fiber of $\widetilde{M}$. 
\end{proposition}

To prove this proposition, we use the following lemma describing the local behavior. 

\begin{lemma}
\label{lem:rotation}
Given a Seifert cover $f: \widetilde{N} \to N$ of Seifert neighborhoods, the covering map is equivalent (as covering spaces) to one whose deck transformation groups acts as rotation on both coordinates of $\partial \widetilde{N}$. Furthermore, $f$ is determined (up to covering space isomorphism) by this action on the boundary. 
\end{lemma}
\begin{proof}
The map $f$ is a covering map with cyclic deck transformation group $G$ since $N$ is homotopy equivalent to a circle. Pick a generator $g$ of $G$.
The generator $g$ acts on $\widetilde{N}$ taking fibers to fibers and has finite order, so it decomposes into an action $g_1$ on the central fiber, $S^1$, and an action $g_2$ on $D^2$, a disk transverse to each fiber. By classification of 1-manifolds $g_1$ is conjugate to a rotation, and by \cite{K}, $g_2$ is conjugate to a rotation, so up to isomorphism of covering spaces, $g$ rotates $\widetilde{N}$ on both coordinates. 
\end{proof}

We are now ready to prove Proposition \ref{prop:pullback}.
\begin{proof}[Proof of Proposition \ref{prop:pullback}]
First, given a Seifert cover, we describe the induced cover on base orbifolds. Consider a Seifert neighborhood $N_{p'/q'}$ of a fiber $\gamma$ in $M$. Each connected component of $f^{-1}(\gamma)$ is a Seifert neighborhood by construction of the Seifert structure on $\widetilde{M}$. It is also clear that if $\gamma$ is a regular fiber, then so is each connected preimage of $\gamma$ since in a regular Seifert neighborhood every fiber generates $\pi_1$. Now quotienting by the $S^1$ action induces homeomorphisms $D^2 \to D^2$ so that $f$ induces a cover between base orbifolds near smooth points. If $\gamma$ is instead a singular fiber with nearby fibers homotopic to $k$ times $\gamma$, then a connected component $\widetilde{\gamma}$ of $f^{-1}(\gamma)$ will have nearby fibers homotopic to $k/d$ times $\widetilde{\gamma}$, where $d$ is the degree of the cover $\widetilde{\gamma} \to \gamma$, by Lemma \ref{lem:rotation}. Indeed, the fibers near $\gamma$ generate $k\Z \subset \Z = \pi_1(\gamma)$, so the fibers near $\widetilde{\gamma}$ must generate $k/d \Z \subset \Z = \pi_1(\widetilde{\gamma})$. Thus we have an induced map of base orbifolds $D^2(k/d) \to D^2(k)$ by the obvious quotient, so that $f$ induces a cover on base orbifolds near singular fibers as well. 

Now, let $\underline{f}: \widetilde{\Sigma} \to \Sigma$ be the induced cover of base orbifolds, let $\rho: M \to \Sigma$ be the projection, and define 
\[
\overline{M} := \{ (m, \widetilde{s}) | m \in M, \widetilde{s} \in \widetilde{\Sigma} , \rho(m) = \underline{f}(\widetilde{s})\}.
\]
Now it is easy to check that the projection $f_1: \overline{M} \to M$ given by $f_1(m, \widetilde{s}) = m$ is a cover of the same degree as $(\underline{f})$, and that lifting the Seifert fiber structure on $M$ to $\overline{M}$ makes $\overline{M}$ a pullback cover of $M$. Similarly, the map $f_2: \widetilde{M} \to \overline{M}$ given by $f_2(\widetilde{m}) = (f(\widetilde{m}),\rho(\widetilde{m}))$ is a fiberwise cover since by construction it induces the identity map on base orbifolds. 

\end{proof}

It will also be useful to describe explicitly the effect of fiberwise and pullback covers on the standard Seifert fiber form, which is stated in the following two corollaries.

\begin{corollary}
\label{cor:fiberwisecov}
Let $f: \widetilde{M} \to M$ be a fiberwise cover with $\widetilde{M} = \{b;(\alpha_1, \beta_1), \dots, (\alpha_k, \beta_k)\}$. Then $M = \{d_f b;(\alpha_1, d_f \beta_1), \dots, (\alpha_k,d_f \beta_k)\}$, where $d_f$ is the degree of $f$.
\end{corollary}
\begin{proof}
Begin by rewriting $\widetilde{M}$ as $\{0;(\alpha_1, \beta_1), \dots, (\alpha_k, \beta_k), (\alpha_{k+1},b)\}$ with $\alpha_{k+1} = 1$. Then applying Proposition \ref{prop:pullback} to a neighborhood of each listed fiber gives the result. See figure \ref{fig:fibcov}.

\begin{figure}

\scalebox{1.4}{
\begin{tikzpicture}
\tikzset{->-/.style={decoration={
  markings,
  mark=at position .5 with {\arrow{>}}},postaction={decorate}}}
\draw[->-,>=stealth](0,0) to node[below]{\scalebox{.8}{$\delta$}} (3,0);
\draw[->-,>=latex](3,0) to node[right]{\scalebox{.8}{$\sigma$}} (3,3);
\draw[->-,>=stealth](0,3) to (3,3);
\draw[->-,>=latex](0,0) to (0,3);
\draw(0,0)--(1,3);
\draw(1,0)--(2,3);
\draw(2,0)--(3,3);
\draw[dotted](0,2.25)--(.25,3);
\draw[dotted](.25,0)--(.75,1.5);
\draw[dotted](0,1.5)--(.5,3);
\draw[dotted](.5,0)--(.75,.75);
\draw[dotted](0,.75)--(.75,3);
\draw[dashed](.75,0)--(.75,3);
\draw[dashed](1.5,0)--(1.5,3);
\draw[dashed](2.25,0)--(2.25,3);
\end{tikzpicture}}

\caption{The degree 4 fiberwise quotient of $N_{1/3}(\gamma)$ is $N_{4/3}(f(\gamma))$. $\delta$ is a fiber in $N_{1/3}$, $\sigma$ and the dashed lines are sections of $\partial N_{1/3}$ with the same image in $N_{4/3}$, the solid diagonal line is a meridian of $N_{1/3}$, and the dotted line is its image in $N_{4/3}$.}
\label{fig:fibcov}

\end{figure}
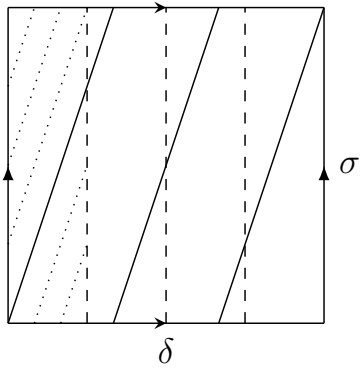

\end{proof}

\begin{corollary}
Let $f: \widetilde{M} \to M$ be a pullback of base orbifolds with 
\[
M = \{b;(\alpha_1, \beta_1), \dots, (\alpha_k, \beta_k)\}.
\]
Then 
\[
\widetilde{M} = \{db;\bigg{(} \frac{\alpha_1}{\lambda_1(\alpha_1)}, \beta_1 \bigg{)}, \dots, \bigg{(}\frac{\alpha_1}{\lambda_{r_1}(\alpha_1)}, \beta_1 \bigg{)}, \dots,  \bigg{(} \frac{\alpha_k}{\lambda_k(\alpha_k)}, \beta_k \bigg{)}, \dots, \bigg{(}\frac{\alpha_k}{\lambda_{r_k}(\alpha_k)}, \beta_k \bigg{)}\}
\]
where $d$ is the degree of $f$, $\lambda(\alpha_i)$ is the partition of $d$ by divisors of $\alpha_i$ coming from the cover of base orbifolds, $\lambda_j(\alpha_i)$ is the $j$th part of the partition $\lambda(\alpha_i)$ (in any order), and $r_i$ is the length of $\lambda(\alpha_i)$.
\end{corollary}
\begin{proof}
The $\alpha$ Seifert invariants are determined by the cone points from the orbifold cover, which are determined from the partitions as stated. See for example \cite[section 1]{EKS}. The $\beta$ Seifert invariants are left unchanged by Proposition \ref{prop:pullback}. Writing the $b$ from $M$ as a $(1,b)$ fiber, this then lifts to $d$-many $(1,b)$ fibers in $\widetilde{M}$ by Proposition \ref{prop:pullback}, which can then be reconsolidated into $db$. See figure \ref{fig:orbpullback}.
\end{proof}

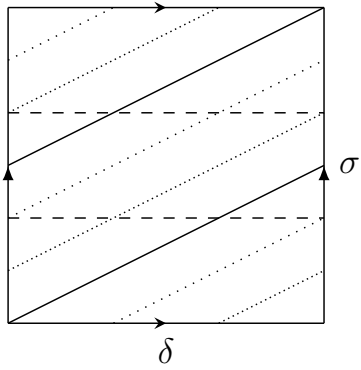
\begin{figure}
\scalebox{1.4}{
\begin{tikzpicture}
\tikzset{->-/.style={decoration={
  markings,
  mark=at position .5 with {\arrow{>}}},postaction={decorate}}}
\draw[->-,>=stealth](0,0) to node[below]{\scalebox{.8}{$\delta$}} (3,0);
\draw[->-,>=latex](3,0) to node[right]{\scalebox{.8}{$\sigma$}} (3,3);
\draw[->-,>=stealth](0,3) to (3,3);
\draw[->-,>=latex](0,0) to (0,3);
\draw(0,0)--(3,3/2);
\draw(0,3/2)--(3,3);
\draw[dotted](1,0)--(3,1);
\draw[dotted](0,1)--(3,2.5);
\draw[dotted](0,2.5)--(1,3);
\draw[densely dotted](2,0)--(3,.5);
\draw[densely dotted](0,.5)--(3,2);
\draw[densely dotted](0,2)--(2,3);
\draw[dashed](0,1)--(3,1);
\draw[dashed](0,2)--(3,2);
\end{tikzpicture}}

\caption{The pullback of $N_{2/3}(\gamma)$ along $f_*:D^2 \to D^2(3)$ is $N_{2/1}(f^{-1}(\gamma))$. 
$\delta$ and the dashed lines are fibers of $N_{2/1}$ with image the same fiber in $N_{2/3}$, and $\sigma$ is a section of $\partial N_{2/1}$. 
The diagonal lines are meridians of $\partial N_{2/1}$ with the same image in $N_{2/3}$. }
\label{fig:orbpullback}

\end{figure}

\section{Orbifold covers\label{sec:orbcov}}

In this section we will classify all orbifold covers of the form $S^2(a,b,c) \to S^2(a',b',c')$. Taking $a' = r$ and $b' = s$, Moser's classification along with Proposition \ref{prop:pullback} will allow us to classify coverings between surgeries on $T(r,s)$.

Since the orbifold Euler characteristic (or just orbifold characteristic, $\chi_{orb}$) is multiplicative under covers, we can further decompose the problem into the cases $\chi_{orb} < 0$, $\chi_{orb} = 0$, and $\chi_{orb} > 0$. These correspond to the three cases in Theorem \ref{thm:orb}.

\subsection{Covers of negative orbifold characteristic}

\begin{proposition}
\label{prop:orbcovlist3}
The only non-trivial covers of orbifolds $S^2(a,b,c) \to S^2(a',b',c')$ with negative orbifold characteristic are
\[
\begin{array}{c|c|c||c|c|c}
(a,b,c) & (a',b',c') & $degree$ &(a,b,c)&(a',b',c')&$degree$ \\
\hline
 (x,x,y)& (2,x,2y)& 2 & (4,4,5)& (2,4,5)& 6 \\
 (2,x,2x)& (2,3,2x)& 3 & (3,3,7)& (2,3,7)& 8 \\
 (x,x,x)& (3,3,x)& 3& (2,7,7)& (2,3,7)& 9 \\
  (3,x,3x)& (2,3,3x)& 4&(3,8,8)& (2,3,8)& 10  \\
 (x,2x,2x)& (2,4,2x)& 4&  (4,8,8)& (2,3,8)& 12 \\
  (x,x,x)& (2,3,2x)& 6& (9,9,9)& (2,3,9)& 12\\
 (x,4x,4x)& (2,3,4x)& 6 & && \\
\end{array}
\]
where $x,y \in \Z$ are large enough that $\chi_{orb} <0$.
\end{proposition}
Observe that since Seifert fiber spaces over these orbifolds have a unique base orbifold \cite[Theorem 5.2]{JN}, the only possible torus knots these covers can occur on are $T(2,x), T(4,5), T(3,7)$ and $T(3,8)$.
\begin{proof}
To begin with, multiplicativity of the orbifold characteristic gives
\[
\frac{1}{a} + \frac{1}{b} + \frac{1}{c} - 1 = n\bigg{(} \frac{1}{a'} + \frac{1}{b'}+ \frac{1}{c'} - 1 \bigg{)}
\]
where $n$ is the degree of the cover. By assumption, $\chi_{orb} < 0$, so both $\frac{1}{a} + \frac{1}{b} + \frac{1}{c} - 1$ and  $\frac{1}{a'} + \frac{1}{b'} + \frac{1}{c'} - 1$ are between 0 and $-1$. We first consider the case $n \geq 7$. In this case $\frac{6}{7} < \frac{1}{a'} + \frac{1}{b'} + \frac{1}{c'} < 1$, and so there are finitely many potential triples $(a',b',c')$. For each of these triples, the partition condition on covers gives a finite list of triples $(a,b,c)$ and degrees $n$ for which we might have a cover $S^2(a,b,c) \to S^2(a',b',c')$. 

Now we associate to each degree $n$ orbifold cover $S^2(a,b,c) \to S^2(a',b',c')$ a cover of $S^1 \vee S^1$ also of degree $n$ in the following way. Split the base $S^2$ into three regions with a wedge of two circles such that each region contains one orbifold point. Then the original cover gives a gluing of some covers of the resulting disk orbifolds onto a cover of $S^1 \vee S^1$. This is shown for $S^2(x,x,y) \to S^2(2,x,2y)$ in figure \ref{fig:negorbcov}. Now since the problem is reduced to covers of degree less than 7, plus some finite number of potential exceptions, we can use a brute force search to obtain the stated list of covers.

\end{proof}

\begin{figure}
\centering
\scalebox{.5}{
\begin{tikzpicture}
\node at (1,0) {\Huge{x}};
\node at (3.9,0) {\Huge{y}};
\node at (6.75,0) {\Huge{x}};
\draw[thick] (1,0) circle [radius=0.75];
\draw[thick] (6.75,0) circle [radius=0.75];
\draw[thick] (6,0) arc [radius=3, start angle=45, end angle= 135];
\draw[thick] (6,0) arc [radius=3, start angle=315, end angle= 225];
\end{tikzpicture}
}
\ \ $\xrightarrow{\mbox{2-fold covers}}$ \ \ 
\scalebox{.5}{
\begin{tikzpicture}
\node at (1,0) {\Huge{x}};
\node at (2.5,0) {\Huge{2y}};
\node at (1.75,1) {\Huge{2}};
\draw[thick] (1,0) circle [radius=0.75];
\draw[thick] (2.5,0) circle [radius=0.75];
\end{tikzpicture}
}
\caption{$S^2(x,x,y)$ 2-fold covers $S^2(2,x,2y)$}
 \label{fig:negorbcov}
\end{figure}
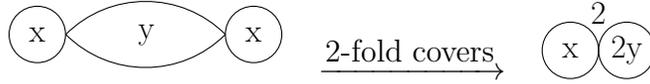

\subsection{Covers of zero orbifold characteristic}

\begin{proposition}
\label{prop:orbcovlist2}
The only covers $S^2(a,b,c) \to S^2(a',b',c')$ with $\chi_{orb} = 0$ are given in Table \ref{table:zoc}.

\end{proposition}

\begin{table}
\[
\begin{array}{c|c|c}
(a,b,c) & (a',b',c') & $degree$ \\
\hline
 (2,3,6)& (2,3,6)& n \\
 (2,4,4)& (2,4,4)& m \\
 (3,3,3)& (3,3,3)& n \\
 (3,3,3)& (2,3,6)& 2n \\
\end{array}
\]

\caption{Covers of zero orbifold charactersitic. $n = x^2 + xy + y^2$ and $m = x^2 + y^2$ for $x,y$ not both $0$.}
\label{table:zoc}
\end{table}

Note that these covers only occur on $T(2,3)$ since the base orbifold of Seifert fiber spaces with these base orbifolds is unique \cite[Theorem 5.2]{JN}.

\begin{proof}
First, recall that the only triples $(a,b,c)$ with $\frac{1}{a} + \frac{1}{b} + \frac{1}{c} = 1$ are $(2,4,4)$, $(3,3,3)$, and $(2,3,6)$. Unlike the other cases, the multiplicativity of the orbifold characteristic tells us nothing about the degree of any potential covers. In particular, each of these orbifolds has many self-covers. The key observation to classify these covers is the connection to lattices. $S^2(2,3,6)$ and $S^2(3,3,3)$ are the fundamental domains of the hexagonal lattice for the p6 and p3 wallpaper groups respectively, and $S^2(2,4,4)$ is the fundamental domain of the square lattice for the p4 wallpaper group. We can then identify covers of these orbifolds with sublattices, keeping track of the symmetries of the sublattice. 

Consider the hexagonal lattice for the $S^2(3,3,3)$ orbifold. That is, a hexagonal lattice with a $\Z/3$ symmetry at each vertex. Any self cover would give a hexagonal sublattice with the same symmetries, and we can identify these sublattices (along with a chosen shortest length vector) with vectors in the original lattice in the following way. Overlay the lattice on $\C$ with 1 corresponding to a shortest length vector. To get a hexagonal sublattice from a vector, multiply each vector in the lattice by the chosen complex number to generate a new lattice, which will induce identical symmetries. See also \cite[section 2.2]{CS}. 

Additionally, neither $S^2(2,3,6)$ or $S^2(2,4,4)$ can cover $S^2(3,3,3)$ since they both have either corner reflectors or a cone point of order 2, neither of which can cover a cone point of order 3. Now the index of the sublattice (and hence the degree of the cover) will be given by the square of the norm of the chosen vector and hence degrees of these self covers are given by outputs of the quadratic form $x^2 + xy + y^2$. See also \cite[Table 4]{CM2}. 

Next consider the hexagonal lattice for $S^2(2,3,6)$. Precisely the same argument will classify self covers. However in this case, for any hexagonal sublattice (where the vertices have a $\Z/6$ rotation action), there is an additional cover corresponding to the same sublattice given by the two fold cover $S^2(3,3,3) \to S^2(2,3,6)$ with partitions $2 = \frac{3}{3} + \frac{3}{3} = \frac{2}{1} = \frac{6}{3}$. That is, corresponding to each hexagonal sublattice, we can forget a 2-fold symmetry and recover $S^2(3,3,3)$. Again, see also \cite[Figure 4]{CM2}. Hence we have $S^2(3,3,3)$ covers $S^2(2,3,6)$ with degree $2(x^2 + xy + y^2)$.

Finally, for $S^2(2,4,4)$ we have a square lattice, and as above, we consider square sublattices with the same symmetries. These have indices $x^2 + y^2$. We also note that these sublattices correspond additionally to covers of $S^2(2,4,4)$ by $S^2(2,2,2,2)$ or by $T^2$ by forgetting additional symmetries.
\end{proof}
\begin{remark}
It is helpful to observe that a priori the degrees of the covers $S^2(3,3,3) \to S^2(2,3,6)$ are of the form $2nn'$ for $n = x^2 + xy +z^2$ and $n' = z^2 + wz + w^2$. However, $nn'$ is again of this form, since compositions of self covers of $S^2(3,3,3)$ must again be self covers of $S^2(3,3,3)$.
\end{remark}
 \begin{remark}
In all of these cases, covers of a specified degree are not necessarily unique. For example $49 = 7^2 + 7 \cdot 0 + 0^2 = 5^2 + 5 \cdot 3 + 3^2$, and hence there are two inequivalent self covers of $S^2(2,3,6)$ of degree 49. 
\end{remark}

\subsection{Covers of positive orbifold characteristic} 

\begin{proposition}
\label{prop:orbcoverlist}
The only non-trivial covers of orbifolds $S^2(a,b,c) \to S^2(a',b',c')$ with positive orbifold characteristic are the following.
\[
\begin{array}{c|c|c|c||c|c|c}
(a,b,c) & (a',b',c') & degree& $conditions$ &(a,b,c)&(a',b',c')&degree\\
\hline
(1, x, y) & (1, nx, ny) & n &&(2,3,3) & (2,3,4) & 2 \\
(1,d,d) & (2,2,x) & 2x/d & d|x &(2,2,4) & (2,3,4) & 3 \\
(1,d,d) & (2,3,3) & 12/d & d \in \{1,2,3\} &(2,2,3) & (2,3,4) & 4\\
(2,2,2) & (2,3,3) &3&&(2,2,2) & (2,3,4) & 6\\
(1,d,d) & (2,3,4) & 24/d & d\in \{1,2,3,4\} &(2,3,3) & (2,3,5) & 5 \\
(1,d,d) & (2,3,5) & 60/d & d  \in \{1,2,3,5\} &(2,2,5) & (2,3,5) & 6 \\
(2,2,d) & (2,2,x) &  x/d  &d|x&(2,2,3) & (2,3,5) & 10\\
 & & &&(2,2,2) & (2,3,5) & 15 \\
\end{array}
\]
Here $n,x,y$ are any positive integers.
Note that since Seifert fiber spaces over these orbifolds (i.e. lens spaces) do not necessarily have unique base orbifolds, these covers may (and in fact do) occur on $T(3,4), T(2,x)$ and $T(3,5)$ in addition to $T(2,3)$.
\end{proposition}

\begin{proof}
Using only multiplicativity of orbifold characteristic and the classification of elliptic 2-orbifolds (see for example \cite[section 13.3]{T}), the potential covers are
\begin{enumerate}
\item $S^2(x,y) \to S^2(nx,ny)$
\item $S^2(d,d) \to S^2(2,2,x)$ with $d|x$,
\item $S^2(d,d) \to S^2(2,3,3)$ where $d|12$,
\item $S^2(d,d) \to S^2(2,3,4)$ where $d|24$,
\item $S^2(d,d) \to S^2(2,3,5)$ where $d|60$,
\item $S^2(2,2,d) \to S^2(2,2,x)$ with $d|x$,
\item $S^2(2,2,2) \to S^2(2,3,3)$,
\item $S^2(2,2,3) \to S^2(2,3,3)$,
\item $S^2(2,2,2) \to S^2(2,3,4)$,
\item $S^2(2,2,3) \to S^2(2,3,4)$,
\item $S^2(2,2,4) \to S^2(2,3,4)$,
\item $S^2(2,2,3) \to S^2(2,3,5)$,
\item $S^2(2,3,3) \to S^2(2,3,4)$, 
\item $S^2(2,2,2) \to S^2(2,3,5)$,
\item $S^2(2,3,3) \to S^2(2,3,5)$,
\item $S^2(2,2,5) \to S^2(2,3,5)$.
\end{enumerate}
Not all of these satisfy the partition condition, so applying that restriction as well gives

\begin{enumerate}
\item $S^2(x,y) \to S^2(nx,ny)$
\item $S^2(d,d) \to S^2(2,2,x)$ with $d|x$, 
\item $S^2(d,d) \to S^2(2,3,3)$,  $d \in \{1,2,3\}$,
\item $S^2(d,d) \to S^2(2,3,4)$, $d \in \{1,2,3,4\}$,
\item $S^2(d,d) \to S^2(2,3,5)$, $d \in \{1,2,3,5\}$,
\item $S^2(2,2,d) \to S^2(2,2,x)$, $d|x$, 
\item $S^2(2,2,d) \to S^2(2,3,4)$, $d \in \{2,3,4\}$
\item $S^2(2,3,3) \to S^2(2,3,4)$,
\item $S^2(2,2,2) \to S^2(2,3,3)$
\item $S^2(2,2,2) \to S^2(2,3,5)$,
\item $S^2(2,2,3) \to S^2(2,3,5)$,
\item $S^2(2,2,5) \to S^2(2,3,5)$,
\item $S^2(2,3,3) \to S^2(2,3,5)$.
\end{enumerate}

In fact these are all orbifold covers, which can be shown in the same way as for the negative orbifold case. This is shown for some cases in figures \ref{fig:orbifoldcovers} and \ref{fig:orbifoldcovers2}. The cases $S^2(d,d) \to S^2(2,3,5)$ for $d \in \{1,2,3,5\}$ and $S^2(d,d) \to S^2(2,3,4)$ for $d \in \{1,2,3\}$ are specifically omitted since they are compositions of other covers on the list. $S^2(x,y) \to S^2(nx,ny)$ corresponds to an $n$-fold cover of a single circle. $S^2(2,2,d) \to S^2(2,2,x)$ is similar to $S^2(x,x,y) \to S^2(2,x,2y)$ from figure \ref{fig:negorbcov}. As a final remark we note that there is not necessarily a unique covering space, or even a unique partition for each entry.  For example with respect to the cover $S^2(2,2) \to S^2(2,2,4)$, we have
\[
4 = \frac{2}{1} + \frac{2}{1} = \frac{2}{1} + \frac{2}{1} = \frac{4}{2} + \frac{4}{2}, \]
but also 
\[
4 = \frac{2}{1} + \frac{2}{2} + \frac{2}{2} = \frac{2}{1} + \frac{2}{1} = \frac{4}{1}. 
\]

\end{proof}

\begin{figure}

\subfloat[]{
\scalebox{.4}{
\begin{tikzpicture}
\node at (1,0) {\Huge{3}};
\node at (3.9,0) {\Huge{2}};
\node at (6.75,0) {\Huge{3}};
\draw[thick] (1,0) circle [radius=0.75];
\draw[thick] (6.75,0) circle [radius=0.75];
\draw[thick] (6,0) arc [radius=3, start angle=45, end angle= 135];
\draw[thick] (6,0) arc [radius=3, start angle=315, end angle= 225];
\end{tikzpicture}
}}\qquad
\subfloat[]{
\scalebox{.4}{
\begin{tikzpicture}
\node at (1,0) {\Huge{3}};
\node at (8.4125,2.929) {\Huge{2}};
\node at (15.825,0) {\Huge{3}};
\draw[thick] (1,0) circle [radius=.75];
\draw[thick] (8.4125,2.929) circle [radius=.75];
\draw[thick] (15.825,0) circle [radius=.75];
\draw[thick] (6,0) - - (10.825,0) - - (8.4125,2.179) - - (6,0);
\draw[thick] (6,0) arc [radius=3, start angle=45, end angle= 135];
\draw[thick] (6,0) arc [radius=3, start angle=315, end angle= 225];
\draw[thick] (10.825,0) arc [radius=3, start angle= 225, end angle= 315];
\draw[thick] (10.825,0) arc [radius=3, start angle= 135, end angle= 45];
\end{tikzpicture}
}}\\
\subfloat[]{
\adjustbox{valign=b}{
\scalebox{.4}{\rotatebox{270}{
\begin{tikzpicture}
\node at (1,0) {\rotatebox{90}{\Huge{3}}};
\node at (10.14,3.05) {\rotatebox{90}{\Huge{2}}};
\node at (10.14,-3.05) {\rotatebox{90}{\Huge{2}}};
\draw[thick] (1,0) circle [radius=.75];
\draw[thick] (10.14,3.05) circle [radius=.75];
\draw[thick] (10.14,-3.05) circle [radius=.75];
\draw[thick] (6,0) - - (9.71,2.4125) - - (9.71,-2.4125) - - (6,0);
\draw[thick] (6,0) arc [radius=3, start angle=45, end angle= 135];
\draw[thick] (6,0) arc [radius=3, start angle=315, end angle= 225];
\end{tikzpicture}
}}}}
\subfloat[]{
\adjustbox{valign=b}{
\scalebox{.4}{\rotatebox{270}{
\begin{tikzpicture}
\node at (1,0) {\rotatebox{90}{\Huge{3}}};
\node at (1.585,4.75) {\rotatebox{90}{\Huge{2}}};
\node at (1.585,-4.75) {\rotatebox{90}{\Huge{2}}};
\draw[thick] (1,0) circle [radius=.75];
\draw[thick] (6,0) - - (9.71,2.4125) - - (9.71,-2.4125) - - (6,0);
\draw[thick] (6,0) arc [radius=3, start angle=45, end angle= 135];
\draw[thick] (6,0) arc [radius=3, start angle=315, end angle= 225];
\draw[thick] (9.71, 2.4125) arc [radius = 6, start angle = -295, end angle = -245];
\draw[thick] (9.71, 2.4125) arc [radius = 6, start angle = 295, end angle = 245];
\draw[thick] (9.71, -2.4125) arc [radius = 6, start angle = -295, end angle = -245];
\draw[thick] (9.71, -2.4125) arc [radius = 6, start angle = 295, end angle = 245];
\draw[thick] (4.67,2.41) - - (-1.5,2.41) - - (1.585,4) - - (4.67,2.41);
\draw[thick] (4.67,-2.41) - - (-1.5,-2.41) - - (1.585,-4) - - (4.67,-2.41);
\draw[thick] (1.585,4.75) circle [radius = .75];
\draw[thick] (1.585,-4.75) circle [radius = .75];
\draw[thick] (-1.5,2.41) arc [radius = 6, start angle = 23.5, end angle = -23.7];
\draw[thick] (-1.5,-2.41) arc [radius = 6, start angle = 203.7, end angle = 156.3];
\end{tikzpicture}
}}}}
\subfloat[]{
\adjustbox{valign=b}{
\scalebox{.4}{\rotatebox{270}{
\begin{tikzpicture}
\node at (-2.6,0) {\rotatebox{90}{\Huge{5}}};
\node at (10.14,3.05) {\rotatebox{90}{\Huge{2}}};
\node at (10.14,-3.05) {\rotatebox{90}{\Huge{2}}};
\draw[thick] (-1.96,2.4125) arc [radius=2.5, start angle=105.5, end angle= 254.5];
\draw[thick] (-1.96,2.4125) arc [radius=2.51, start angle=73.5, end angle= 288];
\draw[thick] (1.75,0) - - (-1.96,2.4125) - - (-1.96,-2.4125) - - (1.75,0);
\draw[thick] (10.14,3.05) circle [radius=.75];
\draw[thick] (10.14,-3.05) circle [radius=.75];
\draw[thick] (6,0) - - (9.71,2.4125) - - (9.71,-2.4125) - - (6,0);
\draw[thick] (6,0) arc [radius=3, start angle=45, end angle= 135];
\draw[thick] (6,0) arc [radius=3, start angle=315, end angle= 225];
\end{tikzpicture}
}}}}
\captionsetup{singlelinecheck=off}
\caption[]{Some orbifold covers from Proposition \ref{prop:orbcoverlist}.  
\begin{enumerate}
\item[(A):] $S^2(2,3,3) \to S^2(2,3,4)$  
\item[(B):] $S^2(2,3,3) \to S^2(2,3,5)$
\item[(C):] $S^2(2,2,3) \to S^2(2,3,4)$
\item[(D):] $S^2(2,2,3) \to S^2(2,3,5)$ 
\item[(E):] $S^2(2,2,5) \to S^2(2,3,5)$ 
\end{enumerate}}
\label{fig:orbifoldcovers}
\end{figure}

\begin{figure}
\subfloat[]{
\scalebox{.3}{
\begin{tikzpicture}
\node at (-3.25,-3) {\Huge{$d$}};
\node at (3,-3) {\Huge{$d$}};
\draw[thick] (-2.25,0) - - (-2.25,-6) - - (-4.25,-10) - - (-4.25,4) - - (-2.25,0);
\draw[thick] (2,0) - - (2,-6) - - (4,-10) - - (4,4) - - (2,0);
\draw[thick] (2,0) arc [radius=3, start angle=45, end angle= 135];
\draw[thick] (2,0) arc [radius=3, start angle=315, end angle= 225];
\draw[thick] (2,-6) arc [radius=3, start angle=45, end angle= 135];
\draw[thick] (2,-6) arc [radius=3, start angle=315, end angle= 225];
\draw[thick] (4,-10) arc [radius=7, start angle=54, end angle= 126];
\draw[thick] (4,-10) arc [radius=7, start angle=306, end angle= 234];
\draw[thick] (4,4) arc [radius=7, start angle=54, end angle= 126];
\draw[thick] (4,4) arc [radius=7, start angle=306, end angle= 234];
\end{tikzpicture}
}}
\subfloat[]{
\scalebox{.3}{
\begin{tikzpicture}
\draw[thick] (6,0) - - (10.825,0) - - (8.4125,2.179) - - (6,0);
\draw[thick] (7,7.5) - - (9.825,7.5) - - (8.4125,6.4) - - (7,7.5);
\draw[thick] (6,0) arc [radius=5, start angle=-15, end angle= 55];
\draw[thick] (6,0) arc [radius=5, start angle=-125, end angle= -195];
\draw[thick] (10.825,0) arc [radius=-5, start angle=15, end angle= -55];
\draw[thick] (10.825,0) arc [radius=-5, start angle=125, end angle= 195];
\draw[thick] (8.4125,2.179) arc [radius=3, start angle= -45, end angle= 45];
\draw[thick] (8.4125,2.179) arc [radius=3, start angle= -135, end angle= -225];
\draw[thick] (7,7.5) arc [radius=3, start angle=70, end angle= 110];
\draw[thick] (7,7.5) arc [radius=3, start angle=290, end angle= 250];
\draw[thick] (9.825,7.5) arc [radius=-3, start angle=70, end angle= 110];
\draw[thick] (9.825,7.5) arc [radius=-3, start angle=290, end angle= 250];
\draw[thick] (5,7.5) - - (4.03,5.4) - - (4.03,11) - - (5,7.5);
\draw[thick] (11.825,7.5) - - (12.795,5.4) - - (12.795,11) - - (11.825,7.5);
\draw[thick] (4.03,11) arc [radius=-12.8, start angle=70, end angle= 110];
\draw[thick] (4.03,11) arc [radius=-12.8, start angle=290, end angle= 250];
\end{tikzpicture}
}}
\subfloat[]{
\adjustbox{valign=b}{
\scalebox{.3}{\rotatebox{270}{
\begin{tikzpicture}
\node at (5.25,0) {\rotatebox{90}{\Huge{2}}};
\node at (-1.7,0) {\rotatebox{90}{\Huge{2}}};
\draw[thick] (-1.7,0) circle [radius=.75];
\draw[thick] (5.25,0) circle [radius=.75];
\draw[thick] (6,0) - - (9.71,2.4125) - - (9.71,-2.4125) - - (6,0);
\draw[thick] (-.95,0) - - (2.76,2.4125) - - (2.76,-2.4125) - - (-.95,0);
\draw[thick] (9.71,2.4125) arc [radius=20, start angle=80, end angle= 100];
\draw[thick] (9.71,2.4125) arc [radius=20, start angle=280, end angle= 260];
\draw[thick] (9.71,-2.4125) arc [radius=20, start angle=80, end angle= 100];
\draw[thick] (9.71,-2.4125) arc [radius=20, start angle=280, end angle= 260];
\end{tikzpicture}
}}}}
\subfloat[]{
\adjustbox{valign=b}{
\scalebox{.3}{\rotatebox{270}{
\begin{tikzpicture}
\node at (-2.6,0) {\rotatebox{90}{\Huge{3}}};
\node at (6.75,0) {\rotatebox{90}{\Huge{3}}};
\draw[thick] (-1.96,2.4125) arc [radius=2.5, start angle=105.5, end angle= 254.5];
\draw[thick] (-1.96,2.4125) arc [radius=2.51, start angle=73.5, end angle= 288];
\draw[thick] (1.75,0) - - (-1.96,2.4125) - - (-1.96,-2.4125) - - (1.75,0);
\draw[thick] (6.75,0) circle [radius=.75];
\draw[thick] (6,0) arc [radius=3, start angle=45, end angle= 135];
\draw[thick] (6,0) arc [radius=3, start angle=315, end angle= 225];
\end{tikzpicture}
}}
}}
\subfloat[]{
\adjustbox{valign=b}{
\scalebox{.3}{\rotatebox{270}{
\begin{tikzpicture}
\node at (-2.6,0) {\rotatebox{90}{\Huge{4}}};
\node at (10.35,0) {\rotatebox{90}{\Huge{4}}};
\draw[thick] (-1.96,2.4125) arc [radius=2.5, start angle=105.5, end angle= 254.5];
\draw[thick] (-1.96,2.4125) arc [radius=2.51, start angle=73.5, end angle= 288];
\draw[thick] (9.71,-2.4125) arc [radius=-2.5, start angle=105.5, end angle= 254.5];
\draw[thick] (9.71,-2.4125) arc [radius=-2.51, start angle=73.5, end angle= 288];
\draw[thick] (1.75,0) - - (-1.96,2.4125) - - (-1.96,-2.4125) - - (1.75,0);
\draw[thick] (6,0) - - (9.71,2.4125) - - (9.71,-2.4125) - - (6,0);
\draw[thick] (6,0) arc [radius=3, start angle=45, end angle= 135];
\draw[thick] (6,0) arc [radius=3, start angle=315, end angle= 225];
\end{tikzpicture}
}}
}}
\captionsetup{singlelinecheck=off}
\caption[]{More orbifold covers from Proposition \ref{prop:orbcoverlist}. Figure (A) is drawn for $x/d = 4$. In general, $d$ would be labeling an $x/d$-gon.  
\begin{enumerate}
\item[(A):] $S^2(1,d,d) \to S^2 (2,2,x)$, $x = 4d$ 
\item[(B):] $S^2(1,d,d) \to S^2 (2,3,3)$, $d = 1$ 
\item[(C):] $S^2(1,2,2) \to S^2(2,3,3)$
\item[(D):] $S^2(1,3,3) \to S^2 (2,3,3)$
\item[(E):] $S^2(1,4,4) \to S^2(2,3,4)$
\end{enumerate}}

\label{fig:orbifoldcovers2}

\end{figure}

\begin{proof}[Proof of Theorem \ref{thm:orb}]
This is now a direct consequence of Propositions \ref{prop:orbcovlist3}, \ref{prop:orbcovlist2}, and \ref{prop:orbcoverlist}.
\end{proof}

\section{Realization of orbifold covers \label{sec:realize}}

Now that we have a complete list of possible base orbifold covers, we aim to understand when these covers are realized by Seifert covers of surgeries on a torus knot. By Proposition \ref{prop:pullback} we can split this problem into two parts. First, given a Seifert fiber space $M = S^3_{p/q}(K)$ with base orbifold $\Sigma$ and $\widetilde{\Sigma} \to \Sigma$ a non-trivial cover of orbifolds, when is the pullback of $M$ along this cover also realized by surgery on $K$? We discuss this in \ref{subsection:pullbacks}. Second, given a fixed base orbifold $\Sigma$, which coverings of Seifert fiber spaces occur over $\Sigma$ as surgery on the same torus knot? We discuss this in \ref{subsection:sameorbifold}. Finally, composing a fiberwise cover and a pullback cover may be realized even if the intermediate cover is not. An example is given in \ref{subsection:combo}.

\subsection{Realization of pullbacks of orbifold covers\label{subsection:pullbacks}}

\begin{lemma} 
\label{lem:sphereobs}
Pullbacks along the following coverings of base 2-orbifolds do not occur for surgeries on any torus knot.
\begin{enumerate}
\item $S^2(d,d) \to S^2(2,s,2)$ where $d|s$ and $s$ is odd,
\item $S^2(d,d) \to S^2(2,3,3)$ where $d \in \{1,2,3\}$,
\item $S^2(d,d) \to S^2(2,3,4)$ where $d \in \{1,2,3,4\}$,
\item $S^2(d,d) \to S^2(2,3,5)$ where $d \in \{1,2,3,5\}$.
\end{enumerate}
\end{lemma} 
\begin{proof}
We first consider (1). By Moser's classification $S^2(2,s,2)$ can only occur as a base orbifold from surgery on the torus knot $T(2, s)$. We will check that $S^2(d,d)$ never occurs from surgery on this knot. Since Seifert fiber spaces over $S^2(d,d)$ are lens spaces, Moser's classification implies $|2sq-p|=1$ in the cover. In particular $p \equiv \pm1$ mod $2s$. Computing $p$ (the order of $H_1$) from the Seifert invariants however, gives
\[
p = \pm |H_1(\{b;(d,\beta_1),(d,\beta_2) \} )| = d^2b + d\beta_1 + d\beta_2 \equiv 0 \mbox{ mod } d.
\]
Hence $p \not\equiv \pm1$ mod $2s$ unless (potentially) $d = 1$. In this case we would have the space
\[
\{b ; (2, \beta_1),(s,\beta_2),(2,\beta_3) \}
\]
lifting to 
\[
\{2sb; (1,s\beta_1),(1,2\beta_2),(1,s\beta_3)\} = L(s(2b+\beta_1+\beta_3) + 2\beta_2,1).
\]
In particular then, we would have $p = s(2b+\beta_1+\beta_3) + 2\beta_2 \not\equiv \pm1$ mod $2s$ since it is even. Cases (2)-(4) are similar with the same kind of modular arithmetic obstructions.
\end{proof}
\begin{remark}
While pullbacks along these covers do not occur from surgeries on a torus knot, more general covers which induce these covers of base orbifolds may.
\end{remark}
In contrast to the case of Lemma \ref{lem:sphereobs}, in other cases pullbacks along covers of base orbifolds are often realized as surgeries.
\begin{example}
\label{lemma:k}
Given a surgery with one of the base orbifolds listed below, the pullback along the listed cover is often also a surgery on that torus knot.
\begin{enumerate}
\item $S^2(2,s,s) \to S^2(2,s,4)$ on T(2,$s$),
\item $S^2(2,2,3) \to S^2(2,3,4)$ on T(2,$3$),
\item $S^2(2,3,3) \to S^2(2,3,5)$ on T(2,$3$),
\item $S^2(2,2,3) \to S^2(2,3,5)$ on T(2,$3$),
\item $S^2(2,2,5) \to S^2(2,3,5)$ on T(2,$5$).
\end{enumerate}

First consider (1). Then we have as a base space
\[
\{b; (2,1), (s, \beta_2), (4, \beta_3)\},
\]
where $\beta_3 \in \{1,3\}$. This lifts along the degree 2 cover (1) with corresponding partitions $2 = \dfrac{4}{2} = \dfrac{2}{1} = \dfrac{s}{s} + \dfrac{s}{s}$ to give 
\[
\{2b; (1,1),(s,\beta_2),(s,\beta_2),(2,\beta_3)\} = \{2b + 1; (s,\beta_2),(s,\beta_2),(2,\beta_3)\}.
\]
In particular,
\[
p = \pm|H_1(\{2b + 1; (s,\beta_2),(s,\beta_2),(2,\beta_3)\})| = 2s^2(2b+1) + s^2\beta_3+4s\beta_2 \equiv s \mbox{ mod } 2s.
\]
By Moser's classification this base orbifold is realized whenever $|2sq-p| = s$. In fact for any choice of $p \equiv s$ mod $2s$,  there is a choice of $q$ so that $|2sq-p| = s$. Since $p$ determines $b, \beta_2,$ and $\beta_3$ by Lemma \ref{lem:CRT}, this space
\[
\{2b + 1; (s,\beta_2),(s,\beta_2),(2,\beta_3)\}
\]
is realized as surgery on T(2,$s$) as long as $p$ and $q$ are relatively prime. It is easy to check that this often happens. The other cases (2)-(5) are similar.
\end{example}

\subsection{Realization of covers over a fixed orbifold\label{subsection:sameorbifold}} In this case the only possible covers are fiberwise covers, which are determined by Corollary \ref{cor:fiberwisecov}. Note that the $\beta$ invariants for the $r$ and $s$ singular fibers are the same for all surgeries on $T(r,s)$ since these fibers are in the complement of a neighborhood of the knot. Then once you have these two $\beta$ invariants, the third $\beta$ invariant and $b$ are determined by $p$ (see Lemma \ref{lem:CRT}). Hence it is enough to compute $p$ (the order of $H_1$) in the cover, see if surgery with that $p$ can produce the base orbifold in question, and finally check that multiplying the $r$ and $s$ fiber $\beta$ invariants by the degree of the cover leaves them unchanged. In other words, that the degree of the cover is 1 mod $rs$. We provide an example: 

Consider the Seifert fiber space obtained by $-2/3$ surgery on T(2,5). This has base orbifold $S^2(2,5,32)$ with $H_1$ of order $2$. The standard Seifert form is therefore
\[
\{-2;(2,1),(5,3),(32,29)\}.
\]
Taking this as a degree $d$ fiberwise cover gives 
\[
\{-2d;(2,d),(5,3d),(32,29d)\}.
\]
Now we notice that we must have $d \equiv 1$ mod $2$ and $d \equiv 1$ mod $5$, or $d \equiv 1$ mod $15$ for this new manifold to still be surgery on T(2,5). Furthermore, $H_1$ has order $2d$, which will only be $p/q$ surgery on T(2,5) when $|10q-p| = 32$ and $2d = |p|$.  Additionally, the value of $q$ is then determined by $|10q - p| = 32$, and must be relatively prime to $p$. For example $p = -12, q = 2$ is a solution, but not a valid surgery, whereas $p = -22, q = 1$ is. 

\begin{remark}
This example agrees with \cite[Theorem 1.12]{LM}, since although $2/3 <1$, $\lceil 3/2 \rceil > \lfloor 1/22 \rfloor$ so this (regular) cover is consistent with their theorem.
\end{remark}

\subsection{Realization of compositions of covers\label{subsection:combo}}
We describe the general method for checking if one Seifert fiber space $\widetilde{M}$ covers another Seifert fiber space $M$, according to Proposition \ref{prop:pullback}.
\begin{enumerate}
\item First check if there exists a cover between the base orbifolds. Note that $M$ comes with a specified base orbifold, but if $\widetilde{M}$ is a lens space, then we must check all $S^2(d,d)$ which cover the base orbifold of $M$. For small Seifert fiber spaces this is classified in section \ref{sec:orbcov}.

\item Next compute the pullback of the proposed base manifold $M$ along the cover of base orbifolds from (1), as described in section \ref{subsection:pullbacks}

\item Finally check if the proposed cover $\widetilde{M}$ covers this pullback as described in section \ref{subsection:sameorbifold}.
\end{enumerate}

\begin{proof}[Proof of Theorem \ref{thm:torusmain}.]

By Lemma \ref{lem:lensspaces}, we can reduce to the case that at least one of the two surgeries is not a lens space. Theorem \ref{thm:orb} classifies covers of base orbifolds in this case. All such non-trivial covers could only occur on the listed exceptional torus knots, so the remaining covers are fiberwise covers. 

Suppose that $\widetilde{M} \to M$ is a degree $d$ fiberwise cover, and 
\[
\widetilde{M} = \{b;(r, \beta_1), (s, \beta_2), (\alpha,\beta_3)\}.
\]
Then according to \cite{M}, $|H_1(\widetilde{M})| = |r s \alpha b + rs \beta_3 + r \beta_2 \alpha + \beta_1 s \alpha|$, and by Corollary \ref{cor:fiberwisecov},
\[
M = \{db;(r, d\beta_1), (s,d \beta_2), (\alpha,d\beta_3)\}.
\]
This gives that $|H_1(M)| = |rs \alpha d b + rs d\beta_3 +r d\beta_2 \alpha + d\beta_1 s \alpha| = d \cdot |H_1(\widetilde{M})|$. Additionally, we must have that $\beta_1 \equiv d\beta_1$ mod $r$ and $\beta_2 \equiv d \beta_2$ mod $s$, and since gcd$(r,\beta_1) = $ gcd$(s, \beta_2)$ this is equivalent to $d \equiv 1$ mod $rs$.

Conversely, once we fix $r,s,\beta_1,$ and $\beta_2$, $|H_1(M)|$ and the base orbifold determine $M$. Hence as long as $|rsq-p| = |rsq' - p'|$, we can try to take an appropriate degree fiberwise cover of $S^3_{p'/q'}(T(r,s))$ to get $S^3_{p/q}(T(r,s))$. This cover will exist if and only if $p'|p$ and $p/p'$ is relatively prime to the indices of the singular fibers, $r, s,$ and $|rsq'-p'|$.

\end{proof}

We conclude with a pair of examples. 
\begin{example}
\label{ex:cover}
Let $\widetilde{M}$ be $(5,1)$ surgery on $T(2,3)$ and let $M$ be $(45,7)$ surgery on $T(2,3)$. Then by Moser's classification $M$ is given by
\[
\{1;(2,1),(3,1),(3,2)\}
\]
with base orbifold $S^2(2,3,3)$. Since $\widetilde{M}$ is a lens space, we should check pullbacks along $S^2 \to S^2(2,3,3)$, $S^2(2,2) \to S^2(2,3,3)$, and $S^2(3,3) \to S^2(2,3,3)$. We will first pull back along $S^2(3,3) \to S^2(2,3,3)$, which will turn out to be sufficient. The partitions for this degree 4 cover are $4 = \dfrac{2}{1} + \dfrac{2}{1} = \dfrac{3}{1} + \dfrac{3}{3} = \dfrac{3}{1} + \dfrac{3}{3}$ as computed from figure \ref{fig:orbifoldcovers2}. This gives the Seifert fiber space
\[
\{4;(1,1),(1,1),(1,1),(3,1),(1,2),(3,2)\} = \{9;(3,1),(3,2)\} = L(90,-29).
\]
This is 19-fold covered by $L(5,-29) = L(5,1)$, which by Moser's classification is $M$. In fact no cover $\widetilde{M} \to M$ could come from a cover of the complement of $T(2,3)$, since such a cover would necessarily be fiber preserving on the knot complement. Alternatively, since the complement of $T(2,3)$ is also Seifert fibered (with Seifert invariants $(2,1),(3,\pm1)$, depending on orientation), it is also possible to compute all self covers directly. 

\end{example}

\begin{example}
\label{ex:cover2}
Let $\widetilde{M}$ be $105/4$ surgery on $T(4,7)$ and let $M$ be $21/1$ surgery on $T(4,7)$. Then by Theorem \ref{thm:torusmain} $\widetilde{M}$ is a 5-fold cover of $M$, both of which have base orbifold $S^2(4,7,7)$. However, this cover does not restrict to a self cover of the $T(4,7)$ complement, as can be seen from the Seifert invariants.
\[
\widetilde{M} = \{-1;(4,1),(7,5),(7,4)\}, \ \ M = \{-1;(4,1),(7,5),(7,1)\}.
\]
The degree 5 cover between them sends the $(7,5)$ fiber to the $(7,1)$ fiber, whereas in a self cover of the knot complement, the $(7,5)$ fiber must be preserved.
\end{example}

\section{Hyperbolic Knots \label{sec:hyp}}

In this section we will first use a theorem of Futer, Kalfagianni, and Purcell to prove Proposition \ref{prop:hyp}, and then we will use computations of the hyperbolic volume and identification of exceptional surgeries to prove Proposition \ref{prop:hyp2}. 

First we will give a necessary definition. For more background information see \cite{R}. We will use the homological framing for knots in $S^3$, so that the longitude refers to the framing curve having linking number 0 with the knot. Using the standard identification of the boundary of a horoball neighborhood of the cusp with a torus quotient of $\C$, we can define complex lengths for the longitude and meridian. These are only determined up to scaling the horoball, so we use the following.

\begin{definition}
The \emph{cusp shape} $s \in \C$ of a hyperbolic knot is $s = l/m$, where $l$ is the complex length of the longitude, and $m$ is the complex length of the meridian. 
\end{definition}

This is independent of the choice of horoball since the longitude and meridian scale together. 

Our first goal is to prove Proposition \ref{prop:hyp}, here restated as Corollary \ref{cor:bounds}, which is a corollary of the following theorem of Futer, Kalfagianni, and Purcell.
\begin{theorem}\cite[Theorem 1.1]{FKP}
\label{thm:HK}
Let $K$ be a hyperbolic knot in $S^3$, and let $l$ be the length of a surgery slope $p/q$ on the knot complement which is greater than $2\pi$. Then
\[
\vol(K_{p/q}) \geq \bigg{(} 1 - \bigg{(}\dfrac{2\pi}{l_{p/q}}\bigg{)}^2\bigg{)}^{3/2}\cdot \vol(S^3 - K).
\]
\end{theorem}

\begin{corollary}
\label{cor:bounds}
Let $K \subset S^3$ be a hyperbolic knot, and $p/q \in \Q$. Then there are at most 32 $p'/q' \in \Q$ such that $K_{p'/q'}$ is non-trivially covered by $K_{p/q}$. 
\end{corollary}

\begin{remark}
A somewhat similar theorem of Hodgson and Kerckhoff \cite[Theorem 5.9, Corollary 6.7]{HK} gives a similar result, but with a bound of 60 surgeries.
\end{remark}

\begin{proof}[Proof of Corollary \ref{cor:bounds}] We will use Theorem \ref{thm:HK} to bound from above the surgery length of hyperbolic surgeries which could contradict the conjecture. Let $\vol(K_{p/q})$ be the hyperbolic volume of $K_{p/q}$, and let $K_{p/q} \to K_{p'/q'}$ be a degree $n$ cover. Since hyperbolic volume is multiplicative under covers (see for example \cite[Theorem 11.6.3]{R}), 
\[
\vol(K_{p/q}) = n\vol(K_{p'/q'}).
\]
Furthermore a theorem of Thurston \cite[Theorem 6.5.6]{T} gives the inequality $ \vol(S^3 - K) > \vol(K_{p/q}), \vol(K_{p'/q'})$. 
Hence by non-triviality of the cover, 
\begin{equation}
\label{eqn:volineq}
\vol(K_{p'/q'}) < \vol(S^3 - K)/2. 
\end{equation}
Now we can solve for $l_{p'/q'}$ in Theorem \ref{thm:HK} to get
\[
l_{p'/q'} < \dfrac{2\pi}{\sqrt{1-(1/2)^{2/3}}} = 10.328942 \dots
\]
We claim there are at most 32 $p'/q'$ for which this is satisfied. Let $p'/q'$ and $r/s$ be slopes such that the above equation is satisfied, and let area$(T)$ be the area of the cusp torus $T$ for $K$. Then as in the proof of \cite[Theorem 8.1]{A},
\[
|p's - rq'| < \dfrac{(10.33)^2}{\mbox{area}(T)}.
\]
Furthermore, area$(T) \geq 2\sqrt{3}$ (see for example \cite{CM}, note that equality holds if and only if $K$ is the knot $4_1$). Combining these results then gives
\[
|p's - rq'| <30.84.
\]
But by \cite[Lemma 8.2]{A}, there are at most $P(k)+ 1$ slopes with intersection number at most $k$ where $P(k)$ is the smallest prime larger than $k$, so there are at most 32 $p'/q'$ such that  $K_{p'/q'}$ is non-trivially covered by $K_{p/q}$.

\end{proof}

The rest of this section is devoted to checking that none of the 32 potential exceptions for low crossing number knots give rise to counterexamples. We proceed by using the computer program SnapPy \cite{SnapPy} to check the hyperbolic surgeries. First, SnapPy will compute the cusp shape $s \in \C$ of a hyperbolic knot. From this it is easiest to compute the normalized surgery length, so we normalize the cusp to have area 1, and to have positive real meridian. Computing this normalized meridian $m$ and longitude $l$ in terms of the cusp shape $s$ given by SnapPy gives
\[
m = \frac{1}{\sqrt{|\mbox{Im}(s)|}}, \ \ \
l = sm.
\]

The following lemma will then let us bound which $p/q$ may give rise to the 32 potentially exceptional surgeries.

\begin{lemma}
\label{lem:bounds}
Let $k \in \R_{>0}$, $a = \dfrac{|k \cdot \mbox{Re}(l)|}{|m\cdot \mbox{Im}(l)|} + \dfrac{k}{m}$ and $b = \dfrac{k}{|\mbox{Im}(l)|}$, and suppose either $|p|>a$ or $|q|>b$. Then $(p,q)$ surgery on $K$ has surgery curve of normalized length greater than $k$.
\end{lemma}

\begin{proof}
The normalized surgery length is $|pm + ql|$, and since $m$ is real, $|q \cdot \mbox{Im}(l)| \leq |pm + ql|$. In particular, as long as $|q| > \dfrac{k}{|\mbox{Im}(l)|}$ then $|pm + ql| > k$. Now suppose $|q| \leq \dfrac{k}{|\mbox{Im}(l)|}$, but that $|p| > \dfrac{|k \cdot \mbox{Re}(l)|}{|m\cdot \mbox{Im}(l)|} + k$. Then
\[
|pm + ql| \geq |\mbox{Re}(pm+ql)| = |\mbox{Re}(pm) + \mbox{Re}(ql)| = |pm + \mbox{Re}(ql)|.
\]
But $|$Re($ql)|$ is at most $\dfrac{k \cdot |\mbox{Re}(l)|}{|\mbox{Im}(l)|}$, so as long as $|pm|$ is at least $\dfrac{|k \cdot \mbox{Re}(l)|}{| \mbox{Im}(l)|} + k$ then $|pm+ql| > k$, or equivalently as long as $|p| \geq \dfrac{|k \cdot \mbox{Re}(l)|}{|m\cdot \mbox{Im}(l)|} + \dfrac{k}{m}$, then $|pm + ql| > k$, as desired.
\end{proof}

Now we can use Lemma \ref{lem:bounds} and SnapPy to finish the case of hyperbolic surgeries on knots with 8 or fewer crossings.

\begin{proposition}
Let $K$ be a hyperbolic knot with 8 or fewer crossings. Then there is no pair of hyperbolic surgeries $S^3_{p/q}(K)$ and $S^3_{p'/q'}(K)$ with a non-trivial covering between them.
\end{proposition}

\begin{proof}
Using Corollary \ref{cor:bounds}, it would be enough to check that among the shortest 32 surgery lengths all have hyperbolic volume greater than $\vol(S^3 - K)/2$. The volumes are checked with SnapPy using Lemma \ref{lem:bounds} to ensure that we check at least the 32 shortest curves. 

For all of them except $S^3_{\pm 5/1}(4_1)$ and $S^3_{1/1}(6_1)$, the volume of the surgered manifold is more than half the volume of the knot complement. Hence by Equation \ref{eqn:volineq} they cannot be covered by other surgeries on the same knot. For the remaining two hyperbolic surgeries, we have
\[
\vol(S^3_{5/1}(4_1)) = 0.9813688 \dots \mbox{ and } \vol(S^3_{1/1}(6_1)) = 1.3985088 \dots
\]
whereas
\[
\vol(4_1) = 2.0298832 \dots \mbox{ and } \vol(6_1) = 3.1639632 \dots
\]
For these two surgeries the volume is more than a third the volume of the knot complement. Hence it is enough to check that these two manifolds have no two fold covers. But
\[
|H_1(S^3_{\pm 5/1}(4_1))| = 5, \mbox{ and } |H_1(S^3_{1/1}(6_1))| = 1
\]
are both odd, so there are no maps from $H_1 \to \Z/2\Z = S_2$, so there are no two fold covers.
\end{proof}

This leaves the case of exceptional (non-hyperbolic) surgeries on knots with 8 or fewer crossings to which we devote the rest of this section. We first consider alternating knots for which exceptional surgeries are classified in \cite[Corollary 1.2]{IM}. In particular, among alternating hyperbolic knots, only twist knots have more than one exceptional surgery. The Regina software \cite{regina} was used to identify the Seifert fibered and toroidal exceptional surgeries, and the zero-surgeries. The case of the toroidal $\pm 4$-surgery is also worked out in \cite[Section 2]{T2}, and is the union of a twisted interval bundle over the Klein bottle and a torus knot complement. Figure \ref{fig:seiferttab} gives the Seifert fibered surgeries, and figure \ref{fig:toroidaltab} gives the toroidal surgeries. For convenience we use the mirrors of $6_1, 7_2,$ and $8_1$, and since $4_1$ is amphichiral we only list its non-negative surgeries.

\begin{figure}
\begin{tabular}{ c | c | c | c   }
  Twist knot &  $+1$-surgery & $+2$-surgery & $+3$-surgery  \\
  \hline
  &&&\\
  $4_1$ & $\{-1;(2,1),(3,1),(7,1)\}$ & $\{-1;(2,1),(4,1),(5,1)\}$ & $\{-1;(3,1),(3,1),(4,1)\}$\\ &&& \\
  $5_2$ & $\{-1;(2,1),(3,1),(11,2)\}$ & $\{-1;(2,1),(4,1),(7,2)\}$ & $\{-1;(3,1),(3,1),(5,2)\}$\\ &&& \\
  $m6_1$ & $\{-1;(2,1),(3,1),(13,2)\}$ & $\{-1;(2,1),(4,1),(9,2)\}$ & $\{-1;(3,1),(3,1),(7,2)\}$\\ &&& \\
  $m7_2$ & $\{-1;(2,1),(3,1),(17,3)\}$ & $\{-1;(2,1),(4,1),(11,3)\}$ & $\{-1;(3,1),(3,1),(8,3)\}$\\ &&& \\
  $m8_1$ & $\{-1;(2,1),(3,1),(19,3)\}$ & $\{-1;(2,1),(4,1),(13,3)\}$ & $\{-1;(3,1),(3,1),(10,3)\}$\\
\end{tabular}

\caption{The exceptional Seifert fiber surgeries on hyperbolic twist knots with 8 or fewer crossings. The $m$ refers to the mirror of the knot, and for $4_1$ there are the additional $-1,-2,-3$-surgeries since it is amphichiral.}
\label{fig:seiferttab}
\end{figure}

\begin{figure}

\begin{tabular}{ c | c | c  }
  Twist knot &  $0$-surgery & $+4$-surgery   \\
  \hline
  && \\
  $4_1$ & $[A:(1,1)] / \begin{pmatrix} 0 & 1 \\ 1 & -2 \end{pmatrix}$ & $(S^3 - T(2,3)) \cup K_I$ \\&& \\
  $5_2$ & $[A:(2,1)]/ \begin{pmatrix} 0 & 1 \\ 1 & -1 \end{pmatrix}$ & $(S^3 - T(2,3)) \cup K_I$ \\ && \\
  $m6_1$ & $[A:(2,1)]/ \begin{pmatrix} 0 & 1 \\ 1 & -2 \end{pmatrix}$ & $(S^3 - T(2,5)) \cup K_I$ \\ &&\\ 
  $m7_2$ & $[A:(3,2)]/ \begin{pmatrix} 0 & 1 \\ 1 & -1 \end{pmatrix}$ & $(S^3 - T(2,5)) \cup K_I$ \\ &&\\
  $m8_1$ & $[A:(3,1)]/ \begin{pmatrix} 0 & 1 \\ 1 & -2 \end{pmatrix}$ & $(S^3 - T(2,7)) \cup K_I$ \\
\end{tabular}

\caption{The exceptional toroidal surgeries on hyperbolic twist knots with 8 or fewer crossings. $K_I$ refers to the nontrivial interval bundle over the Klein bottle coming from the mapping cylinder of the orientation cover. $[A:(x,y)]$ refers to the Seifert fiber space with base surface the annulus and a single exceptional fiber $(x,y)$. Quotienting by a matrix refers to gluing the two torus boundary components together via that element of the mapping class group. The framing is given by choosing the fiber and a section. As in Figure \ref{fig:seiferttab} the $m$ refers to the mirror of the knot, and there is additionally the $-4$-surgery on $4_1$.}
\label{fig:toroidaltab}

\end{figure}

Covers of Seifert fiber spaces are Seifert fiber spaces, and the multiplicativity of orbifold Euler characteristic gives an obstruction to covers between the surgeries in figure \ref{fig:seiferttab}. We now consider the toroidal surgeries in figure \ref{fig:toroidaltab}. 

\begin{lemma}
\label{lem:transfer}
Let $M$ and $N$ be 3-manifolds. If rank $H_1(M;\R) >$ rank $H_1(N;\R)$ then $N$ cannot cover $M$.
\end{lemma}
\begin{proof}
Suppose $f:N \to M$ is a covering map. Then the transfer homomorphism composed with the induced map $f_*$ on homology induces multiplication by deg$(f)$ on $H_1(M;\R)$, which is an isomorphism. This implies that the transfer homomorphism is injective and hence that rank $H_1(M;\R) \leq $ rank $H_1(N;\R)$.
\end{proof}

By Lemma \ref{lem:transfer}, 0-surgery on a knot can never be covered by any non-zero surgery on a knot. It remains to check that $4$-surgery is not covered by $0$-surgery for twist knots. To do so, we consider the geometric decomposition surface of \cite[Section 1.9]{AFW}. This is similar to the geometric torus decomposition, except that it additionally allows Klein bottles coming from $K_I$ components, as we have in Figure \ref{fig:toroidaltab}. Observe that for $4$-surgery on a twist knot we have a single Klein bottle as the geometric decomposition surface, since torus knot complements admit an $\widetilde{SL_2(\R)}$ geometry (see for example \cite{VVT}). Now by \cite[Theorem 1.9.3]{AFW} this geometric decomposition surface lifts to the geometric decomposition surface of any finite cover. In particular, if $0$-surgery on a twist knot covered $4$-surgery on a twist knot, then it would have a (non-empty) geometric decomposition surface cutting it into pieces which each cover the respective torus knot complement.

However, the geometric decomposition surface for the twist knot $0$-surgeries has at most one torus, since the obvious torus cuts it into a single Seifert fiber space $[A:(x,y)]$. However, by multiplicativity of the orbifold characteristic, $[A:(1,1)]$ does not cover $D^2(2,3) = S^3 - T(2,3)$ (and similarly for the other twist knots we consider). Hence $0$-surgery cannot cover $4$-surgery on these twist knots. In particular,

\begin{proposition}
Conjecture \ref{conj:hyp} is true for alternating knots with 8 or fewer crossings.
\end{proposition}

The final case is that of the non-alternating hyperbolic knots of 8 or fewer crossings, the knots $8_{20}$ and $8_{21}$. 

SnapPy \cite{SnapPy} verifies that all surgeries on the knot $8_{21}$, and all surgeries except the 0, 1, and 2 surgery on the knot $8_{20}$ are hyperbolic. In fact, the volumes of surgeries on the knot $8_{21}$ and of hyperbolic surgeries on the knot $8_{20}$ are all large enough to obstruct any non-trivial covers, as in Corollary \ref{cor:bounds}. 

The knot $8_{20}$ is also the pretzel knot $P(-3,3,2)$, and \cite[Theorem 1.1]{M4}, or Wu \cite[Theorem 1.1]{Wu} can be checked to see that the only toroidal surgery on $8_{20}$ is the 0-surgery. Hence the Seifert fiber space surgeries on $P(-3,3,2)$ are the +1 and +2 surgeries, which are identified by Regina as 
\[
\{-1;(3,1),(4,1),(5,2)\} \mbox{ and } \{-1;(2,1),(4,1),(9,2)\}
\]
respectively. These base orbifolds have orbifold characteristic $-13/60$ and $-5/36$ respectively, so there is no cover between these spaces. This concludes the case of hyperbolic knots with 8 or fewer crossings.

\begin{proposition}
Let $K$ be a hyperbolic knot with 8 or fewer crossings. Then $S^3_{p/q}(K)$ does not non-trivially cover $S^3_{p'/q'}(K)$. In particular, Conjecture \ref{conj:hyp} is true for these knots.
\end{proposition}

This also completes the proof of Proposition \ref{prop:hyp2}.

\bibliography{bibliography}{}
\bibliographystyle{alpha}

\end{document}